\documentclass[a4paper,reqno,11pt]{article}

\usepackage{
amsmath,
amssymb,
amsthm,
xcolor,
amsfonts
}

\usepackage[hidelinks]{hyperref}

\usepackage{geometry}
\newcommand\R{\mathbb{R}}
\newcommand\ep{\varepsilon}
\newcommand\dt{\,dt}
\newcommand\dtau{\,d\tau}
\newcommand\ds{\,ds}
\newcommand\dx{\,dx}
\newcommand\LL{{\mathcal L}}

\newtheorem{theorem}{Theorem}[section]
\newtheorem{proposition}[theorem]{Proposition}
\newtheorem{lemma}[theorem]{Lemma}
\newtheorem{corollary}[theorem]{Corollary}

\theoremstyle{definition}
\newtheorem{definition}[theorem]{Definition}

\newtheorem{assumption}[theorem]{Assumption}

\theoremstyle{remark}
\newtheorem{remark}[theorem]{Remark}

\title{An existence result for dissipative nonhomogeneous hyperbolic \\ equations via a minimization approach}

\author{Lorenzo Tentarelli$^\dagger$, Paolo Tilli$^\ddagger$
\\ \ \\
{\small $^\dagger$Dipartimento di Matematica} \\
{\small Sapienza Universit\`a di Roma } \\
{\small Piazzale Aldo Moro, 5, 00185 Roma, Italy} \\
{\small \texttt{tentarelli@mat.uniroma1.it}}\\ \ \\
{\small  $^\ddagger$Dipartimento di Scienze Matematiche ``G.L. Lagrange'' } \\
{\small Politecnico di Torino } \\
{\small Corso Duca degli Abruzzi, 24, 10129 Torino, Italy} \\
{\small \texttt{paolo.tilli@polito.it}}
}

\begin{document}

\maketitle

\begin{abstract}
We discuss a purely variational approach to the study of a wide class of second order nonhomogeneous dissipative hyperbolic PDEs. Precisely, we focus on the wave-like equations that present also a nonzero source term and a first-order-in-time linear term. The paper carries on the research program initiated in \cite{ST1}, and developed in \cite{ST2,TT}, on the De Giorgi approach to hyperbolic equations.
\end{abstract}

\smallskip



\section{Introduction}
\label{sec-intro}

In this paper, we extend some recent recults on the \emph{De Giorgi approach} to nonhomogeneous hyperbolic PDEs (see \cite{TT}) to the dissipative case.

Precisely, we present a version of this method that allows to study equations having the formal structure:
\begin{equation}
 \label{eq-waved}
 w''(t,x)=-\nabla\mathcal{W}\big(w(t)\big)(x)-\nabla\mathcal{G}\big(w'(t)\big)(x)+f(t,x),\qquad(t,x)\in\R^+\times\R^n,
\end{equation}
with two prescribed initial conditions
\begin{equation}
 \label{eq-ic}
 w(0,x)=w_0(x),\quad\quad w'(0,x)=w_1(x),\qquad x\in\R^n,
\end{equation}
where $\nabla\mathcal{W}$ and $\nabla\mathcal{G}$ are the G\^ateaux derivatives of two given functionals $\mathcal{W}:\mathrm{W}\to[0,\infty)$ and $\mathcal{G}:\mathrm{G}\to[0,\infty)$ ($\mathrm{W}$ and $\mathrm{G}$ being some Banach spaces of functions in $\R^n$ -- typically Sobolev spaces).

The central idea of the De Giorgi (or variational, or minimization) approach is that of finding solutions of hyperbolic Cauchy problems as limits of the minimizers of a proper sequence of functionals of the Calculus of Variations. The original conjecture (see \cite{degiorgi1,degiorgi2} for the statement and \cite{serra,ST1,ST2,tentarelli} for some clarifications) concerned the defocusing NLW equation
\begin{equation}
 \label{eq-nlw}
 w''=\Delta w-|w|^{p-2}w\qquad(p\geq2),
\end{equation}
namely \eqref{eq-waved} with
\[
 \mathcal{W}=\int_{\R^n}\bigg(\frac{1}{2}|\nabla v|^2+\frac{1}{p}|v|^p\bigg)\dx,\qquad\mathcal{G}\equiv0\qquad\text{and}\qquad f\equiv0,
\]
and claimed that the minimizers $(w_\ep)$ of the functionals
\[
 F_{\ep}^h(w):=\int_0^{\infty}e^{-t/\ep}\,\left(\frac{\ep^2}{2}\int_{\R^n}|w''(t,x)|^2\dx+\mathcal{W}\big(w(t,\cdot\,)\big)\right)\dt,
\]
subject to the boundary conditions \eqref{eq-ic}, converge to a solution of the Cauchy problem \eqref{eq-nlw}-\eqref{eq-ic}.

This conjecture has been (essentially) proved in \cite{ST1} (see also \cite{stefanelli}) and then generalized in two consecutive steps: in \cite{ST2} is discussed an abstract version of the problem where \eqref{eq-nlw} is replaced by \eqref{eq-waved} for some general functionals $\mathcal{W}$ and $\mathcal{G}$, but still in the homogeneous case $f\equiv0$; in \cite{TT}, on the other hand, is discussed the case of a nonzero source term $f$, but without dissipative effects (i.e., $\mathcal{G}\equiv0$). In view of this, our paper presents a novely both with respect to \cite{ST2} and with respect to \cite{TT}, thus representing (in some sense) a conclusion of the research program initiated in \cite{ST1}.

It is clear that in the case of the complete equation \eqref{eq-waved},  $F_\ep^h$ is no longer the appropriate functional. To this aim we define
\begin{equation}
 \label{eq-fuc}
 F_{\ep}(w):=F_{\ep}^h(w)+F_{\ep}^d(w)-F_{\ep}^s(w),
\end{equation}
with
\[
 F_{\ep}^d(w):=\ep\int_0^{\infty}e^{-t/\ep}\,\mathcal{G}\big(w'(t,\cdot\,)\big)\dt
\]
(as in \cite{ST2}) and
\[
 F_{\ep}^s(w):=\int_0^{\infty}\int_{\R^n}e^{-t/\ep}\,f_{\ep}(t,x)w(t,x)\dx\dt
\]
(as in \cite{TT}), where $(f_{\ep})$ is a suitable sequence of approximations of $f$ (for details about the reason of the approximating sequence, we refer the reader to \cite{TT}). Such a choice has an immediate heuristic justification: if $w_\ep$ is a minimizer of $F_\ep$ subject to \eqref{eq-ic}, then the Euler-Lagrange equations of $F_\ep$ turn out to be
\begin{multline*}
 \ep^2\,w_{\ep}''''(t,x)-2\ep w_{\ep}'''(t,x)+w_{\ep}''(t,x)=-\nabla\mathcal{W}\big(w_{\ep}(t,\cdot\,)\big)(x)+f_{\ep}(t,x)+  \\[.3cm]
 -\nabla\mathcal{G}\big(w_{\ep}(t,\cdot\,)\big)(x)+\ep\big(\nabla\mathcal{G}\big(w_{\ep}(t,\cdot\,)\big)(x)\big)'.
\end{multline*}
Hence, the link with \eqref{eq-waved} is immediate: as $\ep\downarrow0$, supposing that $f_\ep\to f$ and $w_\ep\to w$, one formally obtains \eqref{eq-waved} (and \eqref{eq-ic}) in the limit.

A first comment is in order. The paper provides a purely variational method for proving existence of (weak) solutions to second order hyperbolic PDEs \emph{with dissipation}, that is, wave-like equations that display extra first-order-in-time terms. This represents a further credit of the De Giorg approach since (to the best of our knowledge) there is no direct variational method that applies to this type of equations (even though the existence of weak solutions for equations like \eqref{eq-waved} is not new in several concrete examples, as one can see for instance in \cite{cherrier,liero2,lions4,lions2,strauss,tao} and references therein). For the sake of completeness, we recall that in the original formulation of De Giorgi (\cite{degiorgi1}) there is no mention to the possibility of extending the conjecture to dissipative equations, while this idea has been first introduced in \cite{ST2}.

\begin{remark}
 Suitable variants of the variational approach of De Giorgi have been developed in the last years in order to study other evolutions equations. For the applications to parabolic equations we mention e.g. \cite{akagi,bogelein,melchionna}, while for the applications to ODE systems we mention \cite{liero} (and references therein). We also quote some new extensions to the Navier-Stokes equation due to \cite{ortiz}.
\end{remark}

The main results of the paper are stated in Theorem \ref{result}, which naturally extends the outcomes both of \cite{ST2} and of \cite{TT}, under the same assumptions on $\mathcal{W}$, $\mathcal{G}$ and $f$. In particular, if we let $f\equiv0$ in Theorem \ref{result}, then we obtain all the results of \cite{ST2}; whereas, setting $\mathcal{G}\equiv0$, we recover \cite[Theorem 2.3]{TT}. Some interesting examples of dissipative nonhomogeneous equations covered by Theorem \ref{result} are present in Section \ref{sec-examples}, which can be seen therefore as an appendix of this introduction.

\medskip
In \cite{ST1,ST2,TT}, the main ingredient to obtain the required a-priori estimates on the minimizers $w_\ep$ is the control (uniform in $\ep$) of the \emph{approximate energy}, i.e.
\begin{equation}
\label{defea}
 \mathcal{E}_{\ep}(t):=\frac{1}{2}\int_{\R^n}|w_{\ep}'(t,x)|^2\dx+\ep^{-2}\int_0^{\infty}s \,e^{-s/\ep} \,\mathcal{W}\bigl(w_{\ep}(t+s)\bigr)\ds,
\end{equation}
that is an approximation (see \cite{TT}) of the energy usually associated with the solutions of \eqref{eq-waved}, i.e.
\begin{equation}
\label{eq-acmec}
 \mathcal{E}(t):=\frac{1}{2}\int_{\R^n}|w'(t,x)|^2\dx+\mathcal{W}\bigl(w(t)\bigr)
\end{equation}
(which is formally preserved when $f\equiv0$ and $\mathcal{G}\equiv0$). The main feature of $\mathcal{E}_\ep$ is the \emph{acausality} (meant as in \cite{tao}). Indeed, the second term of \eqref{defea} depends on all the values of $w_\ep(\tau)$, for every $\tau\geq t$. The meaning and the implications of this property, especially in the case $f\not\equiv0$ and $\mathcal{G}\equiv0$, are extensively explained in \cite{TT}. Here we limit ourselves to recall that the key point that allows to establish the required causal estimates on $(w_\ep)$ is the detection of suitable approximations $f_\ep$ of the forcing term $f$ (while in the homogeneous cases discussed by \cite{ST1,ST2}, this is not necessary due to the monotonicity of $\mathcal{E}_\ep$).

The main difference in the discussion of the complete equation \eqref{eq-waved}, with both dissipation and a source term, is that $\mathcal{E}_\ep$ is no more the crucial quantity. It is necessary to define the modification of $\mathcal{E}_\ep$ given by
\begin{equation}
 \label{defead}
 \mathcal{E}_\ep^d(t):=\mathcal{E}_\ep(t)+\frac{1}{\ep^2}\int_0^t\bigg(\int_0^\infty \frac{e^{-\tau/\ep}\,(\ep+\tau)}{\ep^2}\,\mathcal{G}\big(\ep w_\ep'(\tau+s)\big)\dtau\bigg)\ds,
\end{equation}
which coincides with $\mathcal{E}_\ep$ at $t=0$, but takes into account the dissipative effects. It is clear that such a quantity is not an approximation of the energy $\mathcal{E}$. Nevertheless, if we consider the natural correction of the energy that includes dissipation, namely
\[
 \mathcal{E}^d(t):=\mathcal{E}(t)+2\int_0^t\mathcal{G}\big(w'(s)\big)\ds,
\]
then (exploiting \eqref{eq-defG} and arguing as in \cite{TT}) one formally obtains that $\mathcal{E}_\ep^d(t)\to\mathcal{E}^d(t)$, as $\ep\downarrow0$ (whenever $w_\ep\to w$). In addition, since in the dissipative case $\mathcal{E}^d$ is a formally preserved quantity (up to the action of the external source), namely
\begin{equation}
 \label{eq-meccons}
 \mathcal{E}^d(t)=\mathcal{E}^d(0)+\int_0^t\int_{\R^n}f(s,x)w'(s,x)\dx\ds,\qquad\forall t\geq0,
\end{equation}
then we see that, as in the non-dissipative case, the central question is the investigation of a suitable approximation of the energy-type quantity which is is supposed to be preserved along the flow.

We also point out that $\mathcal{E}_\ep^d$ (as well as $\mathcal{E}_\ep$ in the nondissipative case) does not display any a-priori monotonicity when $f\not\equiv0$. Consequently, the main point of the paper is showing how to adapt the techniques developed in \cite{TT} to establish suitable growth estimates on $\mathcal{E}_\ep^d$. For the sake of simplicity, in our proofs, we will discuss only the new aspects, referring to \cite{TT} (and \cite{ST2}) for those results which do not require significant modifications.

\medskip
Finally, we recall that (as in \cite{ST2,TT}) the full strength of Theorem \ref{result} is obtained under assumption \eqref{eq-Wass}, which forces \eqref{eq-waved} to be semilinear (even though of arbitrary order in space). In addition, we highlight that (as in \cite{ST2}) Theorem \ref{result} works in the assumption that the functional $\mathcal{G}$ be quadratic, thus implying that we can manage only hyperbolic equations with linear dissipative terms, albeit, again, without any prescription on the order in space (see Section \ref{sec-examples}).

\bigskip
\bigskip

\noindent\textbf{Remark on Notation.} If $g=g(t,x)$, we use $g(t)$, or equivalently $g(t,\cdot\,)$, to denote the function of $x$ obtained fixing $t$. We also write $g'$, $g''$ etc. to denote partial differentiation with respect to $t$, while differential operators like $\nabla$, $\Delta$ etc. are referred to the sole space variables. Concerning function spaces, we agree that $L^p=L^p(\R^n)$, $H^m=H^m(\R^n)$ etc., the domain $\R^n$ being understood. Finally, $\langle\cdot\,,\cdot\,\rangle$ denotes duality pairing (function spaces being clear from the context), while $(\cdot\,,\cdot\,)_H$ denotes the inner product in a Hilbert space $H$.

\bigskip
\bigskip
\noindent\textbf{Acknowledgements}

\medskip
\noindent L.T. acknowledges the support of MIUR through the FIR grant 2013 ``Condensed Matter in Mathematical Physics (Cond-Math)'' (code RBFR13WAET).


\section{Functional setting and main results}

The abstract equation \eqref{eq-waved} and the functional \eqref{eq-fuc} are defined in terms of the general functionals $\mathcal{W}$ and $\mathcal{G}$. As in \cite{ST2,TT}, they have to meet the following properties.

\begin{assumption}
 \label{ass-W}
 The functional $\mathcal{W}:L^2\to[0,\infty]$ is weakly lower semicontinuous and its domain, i.e. the set of functions
 \begin{equation}
  \label{eq-domW}
  \mathrm{W}:=\{v\in L^2:\mathcal{W}(v)<\infty\},
 \end{equation}
 is a Banach space such that $C_0^{\infty}\hookrightarrow\mathrm{W}\hookrightarrow L^2$ (with dense embeddings). Moreover, $\mathcal{W}$ is G\^ateaux differentiable on $\mathrm{W}$ and its derivative $\nabla\mathcal{W}:\mathrm{W}\to\mathrm{W}'$ satisfies
 \[
  \|\nabla\mathcal{W}(v)\|_{\mathrm{W}'}\leq C\big(1+\mathcal{W}(v)^{\theta}\big),\qquad\forall v\in\mathrm{W},
 \]
 for some suitable constants $C\geq0$ and $\theta\in(0,1)$.\qed
\end{assumption}

\begin{assumption}
 \label{ass-G}
 The functional $\mathcal{G}:L^2\to[0,\infty]$ is defined by
 \begin{equation}
  \label{eq-defG}
  \mathcal{G}(v)=\left\{
  \begin{array}{ll}
   \tfrac{1}{2}\,a(v,v) & \mbox{if }v\in\mathrm{G}\\[.2cm]
   +\infty & \mbox{if }v\in L^2\backslash\mathrm{G},
  \end{array}
  \right.
 \end{equation}
 where $a:\mathrm{G}\times\mathrm{G}\to\R$ is a symmetric, nonnegative, bounded and bilinear form on a Hilbert space $\mathrm{G}$ endowed with the norm
 \begin{equation}
  \label{eq-Gnor}
  \|v\|_{\mathrm{G}}^2:=\|v\|_{L^2}^2+2\mathcal{G}(v)
 \end{equation}
 and such that $C_0^{\infty}\hookrightarrow\mathrm{G}\hookrightarrow L^2$ (with dense embeddings).\qed

\end{assumption}

\begin{remark}
Assumption \ref{ass-W} is typical of Dirichlet-type functionals like $\mathcal{W}(v)=\Vert \nabla^k v\Vert_{L^p}^p$ (with $p>1$). We refer to Section \ref{sec-examples}
for some examples. Here we just recall that Assumption \ref{ass-W} is additively stable, i.e. if two functionals satisfy Assumption \ref{ass-W}, then so does their sum (for more details on this assumption, see \cite{ST2}).
\end{remark}

\begin{remark}
 Assumption \ref{ass-G} is additively stable, as well as Assumption \ref{ass-W}. In addition, due to its particular form, one can easily check that $\mathcal{G}$ is Fr\'echet differentiable and weakly lower semicontinuous (for further remarks, we refer again the reader to \cite{ST2}).
\end{remark}

Now, we can state the main result of the paper.

\begin{theorem}
 \label{result}
 Let $\mathcal{W},\,\mathcal{G}$ be functionals satisfying Assumptions \ref{ass-W} and \ref{ass-G} (respectively). Let also $w_0\in\mathrm{W}$, $w_1\in\mathrm{W}\cap\mathrm{G}$ and $f\in L_{loc}^2([0,\infty),L^2)$. Then, there exists a sequence $(f_\ep)$, converging to $f$ in $L_{loc}^2([0,\infty);L^2)$ such that:
 \begin{itemize}
  \item[\emph{(a)}] \emph{Minimizers}. For every $\ep\in(0,1)$, the functional $F_{\ep}$ defined by \eqref{eq-fuc} has a minimizers $w_{\ep}$ in the class of functions in $H_{loc}^2([0,\infty);L^2)$ satisfying \eqref{eq-ic}.
  \item[\emph{(b)}] \emph{Estimates}. For every $T>0$ and $\tau\geq0$, there exist two constants $C_T,\,C_{\tau,T}>0$, independent of $\ep$, such that
  \begin{equation}
  \label{eq-apruno}
  \sup_{t\in [0,T]}\int_{\R^n}\big(\,|w_{\ep}'(t,x)|^2+|w_{\ep}(s,x)|^2\,\big)\dx\leq C_{T},
 \end{equation}
 \begin{equation}
  \label{eq-aprdue}
  \int_{\tau}^{\tau+T}\mathcal{W}\big(w_{\ep}(t)\big)\dt\leq C_{\tau,T},\qquad\forall T>\ep,
 \end{equation}
 \begin{equation}
  \label{eq-aprdued}
  \int_0^T\mathcal{G}\big(w_{\ep}'(t)\big)\dt\leq C_T,
 \end{equation}
 and
 \begin{equation}
  \label{eq-aprtred}
  \|w_{\ep}''\|_{L^2([0,T];\mathrm{W}')+L^2([0,T];\mathrm{G}')}\leq C_T.
 \end{equation}
 \item[\emph{(c)}] \emph{Convergence}. Every sequence $w_{\ep_i}$ (with $\ep_i\downarrow0$) admits a subsequence which is convergent in the weak topology of $H_{loc}^1([0,\infty);L^2)$ to a function $w$ that satisfies \eqref{eq-ic} (where the latter is meant as an equality in $(\mathrm{W}\cap\mathrm{G})'$). In addition,
 \begin{equation}
  \label{eq-reg}
  w'\in L_{loc}^{\infty}([0,\infty);L^2),\qquad w''\in L_{loc}^2([0,\infty);\mathrm{W}')+L_{loc}^2([0,\infty);\mathrm{G}')
 \end{equation}
 and (up to subsequences)
 \begin{equation}
  \label{eq-regpd}
  w_{\ep_i}'\rightharpoonup w'\qquad\mbox{in}\quad L^2([0,T];\mathrm{G}).
 \end{equation}
 \item[\emph{(d)}] \emph{Energy inequality}. Letting
 \[
  \mathcal{E}^d(t):=\frac{1}{2}\int_{\R^n}|w'(t,x)|^2\dx+\mathcal{W}(w(t))+2\int_0^t\mathcal{G}\big(w'(s)\big)\ds,
 \]
 there holds
 \begin{equation}
  \label{eq-enineqd}
  \mathcal{E}^d(t)\leq\left(\sqrt{\mathcal{E}^d(0)}+\sqrt{\frac{t}{2}\int_0^t\int_{\R^n}|f(s,x)|^2\dx\ds\,}\right)^{2},\qquad\text{for a.e.}\quad t\geq0.
 \end{equation}
 \item[\emph{(e)}] \emph{Solution of \eqref{eq-waved}}. Assuming, furthermore, that for some real numbers $m>0$, $\lambda_k\geq0$ and $p_k>1$ the functional $\mathcal{W}$ takes the form
 \begin{equation}
  \label{eq-Wass}
  \mathcal{W}(v)=\frac{1}{2}\|v\|_{\dot{H}^m}^2+\sum_{0\leq k<m}\frac{\lambda_k}{p_k}\int_{\R^n}|\nabla^kv(x)|^{p_k}\dx,
 \end{equation}
 then the limit function $w$ satisfies
 \begin{align}
  \label{eq-wavedfweak}
  \int_0^{\infty}\int_{\R^n}w'(t,x)\varphi'(t,x)\dx\dt= &\, \int_0^{\infty}\left(\langle\nabla\mathcal{W}\big(w(t)\big),\varphi(t)\rangle+\langle\nabla\mathcal{G}\big(w'(t)\big),\varphi(t)\rangle\right)\dt+\nonumber\\[.2cm]
  & -\int_0^{\infty}\int_{R^n}f(t,x)\varphi(t,x)\dx\dt
 \end{align}
 for every $\varphi\in C_0^{\infty}(\R^+\times\R^n)$, namely is a \emph{variational} solution of \eqref{eq-waved}.
 \end{itemize}
\end{theorem}

\begin{remark}
 The functional defined in \eqref{eq-Wass} satisfies Assumption \ref{ass-W} with $\mathrm{W}=\{v\in H^m:\nabla^kv\in L^{p_k},\,0\leq k<m\}$ (for details see \cite{ST2}). Recall also that $\|v\|_{\dot{H}^m}$ is the $L^2$ norm of $|\xi|^m\,\widehat{v}(\xi)$, where $\widehat{v}$ is the Fourier transform of $v$. The typical case is $m\in\mathbb{N}$ when $\|v\|_{\dot{H}^m}$ reduces to $\|\nabla^mv\|_{L^2}$.
\end{remark}

Some comments are in order. There are several possible variants of Theorem \ref{result} that one could prove as well (e.g. nonconstant coefficients in \eqref{eq-Wass}, more general lower-order terms, $\R^n$ replaced by domains with Dirichlet or Neumann boundary conditions). For more details we refer the reader to \cite{ST1,ST2,TT}.

In addition, we highlight that it is still an open problem if the assumption on the quadraticity of the highest order term of $\mathcal{W}$ can be removed. Furthermore, it is clear that Assumption \ref{ass-G} entails that only linear dissipative terms are allowed in \eqref{eq-waved} (although with no prescription on the space order).

On the other hand, we note that (as in \cite{TT}) Theorem \ref{result} holds under the sole assumption $f\in L^2_{\text{loc}}([0,\infty);L^2)$ on the source term (which is the natural assumption if one seeks solutions of \eqref{eq-waved} with finite energy -- see e.g.~\cite{cherrier,lions2,TT}).

Finally, estimate \eqref{eq-enineqd} deserves some remarks. As we pointed out in \eqref{eq-meccons}, $\mathcal{E}^d$ is \emph{formally} preserved along the flow. In other words, one would expect that the actual energy $\mathcal{E}$ (defined by \eqref{eq-acmec}) is balanced at each time by the action of the external sources and the dissipation, that is (recalling that $\mathcal{E}^d(0)=\mathcal{E}(0)$) one formally has
\begin{equation}
 \label{eq-form_cons}
 \mathcal{E}(t)-\mathcal{E}(0)=-2\int_0^t\mathcal{G}\big(w'(s)\big)\ds+\int_0^t\int_{\R^n}f(s,x)w'(s,x)\dx\ds.
\end{equation}
However, we are not able to prove \eqref{eq-form_cons}, whose
 validity is a longstanding open problem in the theory of nonlinear
 wave-type equations (the problem is open even for the NLW equation \eqref{eq-nlw} with large enough $p$, see \cite{ST1}).
  The main difficulty is that a formal differentiation of
    \eqref{eq-acmec} with respect to $t$ produces the ``bad term''
\begin{equation}
 \label{eq-meaning}
 \left\langle\nabla\mathcal{W}\big(w(t)\big),w'(t)\right\rangle
\end{equation}
which is not well-defined as a duality pairing
(unless one obtains further regularity for $w(t)$,
such as the integrability of $|w(t)|^{p-1}w'(t)$ in the model case of the NLW equation \eqref{eq-nlw}: this regularity, however, is currently not available).
 Indeed, if enough regularity were available, 
 then one could compute the derivative of $\mathcal{E}$ and (using \eqref{eq-waved} and Assumption \ref{ass-G}) one would obtain
\begin{equation}
 \label{eq-form_der}
 \mathcal{E}'(t)=2\mathcal{G}\big(w'(t)\big)+\int_{\R^n}f(t,x)w'(t,x)\dx.
\end{equation}
Unfortunately, the techniques we developed thus far are not sufficient, in general, neither to prove that the approximate energy $\mathcal{E}_\ep$ satisfies \eqref{eq-form_der} in the limit, nor to establish sufficient regularity on $w'(t)$ to give a rigorous meaning to \eqref{eq-meaning}.

More generally, the regularity of the solutions  is a crucial
 issue in the discussion of the energy conservation
 (see e.g. \cite{struwe}). With this respect the De Giorgi approach,
 which provides an entirely new way to construct  global solutions,
  is still unexplored as concerns their regularity: indeed,
  in principle, the variational solutions thus obtained might
inherit some regularity features from
  those of the minimizers $w_\ep$, which
  could hopefully be proved by adapting classical tools of elliptic
  regularity theory. Of course, the degeneracy in the limit
     of the integral weight $e^{-t/\ep}$ makes it hard to guess
      the precise regularity that may be preserved in the limit,
        and this issue is an open problem which would deserve
        a thorough investigation.

Finally we point out that, despite of the technical difficulties
related to energy conservation,  the estimate that we establish in \eqref{eq-enineqd} is close to being optimal,
as one can easily
check by applying a formal Gr\"onwall-type argument, combined with Jensen inequality, to \eqref{eq-meccons}.


\section{Examples}
\label{sec-examples}

Before carrying on with the proof of Theorem \ref{result}, it is convenient to better explain, through some examples,
 how it can be applied in some  concrete  cases of dissipative hyperbolic
 equations.

\bigskip
\noindent\textbf{1. Telegraph equation.} One of the classical wave equations with dissipation is the nonlinear Telegraph equation
\begin{equation}
 \label{eq-telegraph}
 w''=\Delta w-|w|^{p-2}\,w-w'+f\qquad(p>1),
\end{equation}
which takes the form \eqref{eq-waved} if we let
\begin{equation}
 \label{eq-ex1}
 \mathcal{W}(v)=\int_{\R^n}\left(\frac{1}{2}|\nabla v|^2+\frac{1}{p}|v|^p\right)\dx\qquad\text{and}\qquad\mathcal{G}(v)=\frac{1}{2}\int_{\R^n}|v|^2\dx.
\end{equation}
One can easily check that Assumptions \ref{ass-W} and \ref{ass-G} are fulfilled, with the choices
\[
\mathrm{W}:=H^1\cap L^p, \quad
  \theta:=1-1/\max\{2,p\}, \quad \mathrm{G}=L^2.
\]
In addition, $\mathcal{W}$ is clearly of the form \eqref{eq-Wass}, so that Theorem \ref{result} fully applies. As a consequence, the Cauchy problem \eqref{eq-telegraph}\&\eqref{eq-ic} admits (at least) a distributional solution $w$, which is the limit (up to subsequences) of the minimizers of the functionals $F_\ep$ defined by \eqref{eq-fuc} in view of \eqref{eq-ex1}. In addition, such a \emph{variational solution} $w$ of \eqref{eq-telegraph}-\eqref{eq-ic} also satisfies the energy inequality \eqref{eq-enineqd}.

\bigskip
\noindent\textbf{2. Telegraph with $p$-Laplacian.} Another example is the $p$-Laplacian version of \eqref{eq-telegraph}, namely
\begin{equation}
 \label{eq-tel_plap}
 w''=\Delta_p w-|w|^{q-2}\,w-w'+f\qquad(p,\,q>1,\,p\neq2),
\end{equation}
which is obtained setting
\begin{equation}
 \label{eq-plap}
 \mathcal{W}(v)=\int_{\R^n}\left(\frac{1}{p}|\nabla v|^p+\frac{1}{q}|v|^q\right)\dx
\end{equation}
and $\mathcal{G}$ as in \eqref{eq-ex1}. Again, 
assumptions \ref{ass-W} and \ref{ass-G} are satisfied, now
letting 
\[
\mathrm{W}:=\{v\in L^2:\nabla v\in L^p,\,v\in L^q\},\quad
  \theta:=1-1/\max\{p,q\},\quad  \mathrm{G}:=L^2.
   \]
 Hence, according to (a)--(d) of Theorem \ref{result}, the sequence of functionals $F_\ep$ possesses a sequence of minimizers $w_\ep$ that converge (up to subsequences) to a function $w$, which satisfies the initial conditions \eqref{eq-ic} and the energy inequality \eqref{eq-enineqd}.

However, since clearly in this case $\mathcal{W}$ does not fulfill \eqref{eq-Wass}, item (e) of Theorem \ref{result} is not available. In other words, even though the limit function $w$ is 
   a natural candidate  to solve \eqref{eq-tel_plap}, the theory so far
   developed does not guarantee  that this is the case.

This situation, as mentioned in the Introduction,  occurs for every \emph{quasilinear} equation, in which case item (e) of Theorem \ref{result}
is not available (this should not come as a surprise since,
as is well known, 
global existence for these equations is a longstanding open problem). 
 Anyway, no counterexample to global existence is known, and
 a full extension of the De Giorgi approach to this class of hyperbolic equations would certainly be a major breakthrough.

\bigskip
\noindent\textbf{3. Strongly damped wave equations.} These equations are characterized by the presence of the term $\Delta w'$ (see e.g.  \cite{kalantarov} for more details). For instance,
\begin{equation}
 \label{eq-stronguno}
 w''=\Delta w-|w|^{p-2}\,w+\Delta w'+f\qquad(p>1),
\end{equation}
can be obtained letting
\begin{equation}
 \label{eq-strongdamp}
 \mathcal{W}(v)=\int_{\R^n}\left(\frac{1}{2}|\nabla v|^2+\frac{1}{p}|v|^p\right)\dx\qquad\text{and}\qquad\mathcal{G}(v)=\frac{1}{2}\int_{\R^n}|\nabla v|^2\dx,
\end{equation}
in \eqref{eq-waved}, while its $p$-laplacian version
\begin{equation}
 \label{eq-strongdue}
 w''=\Delta_p w-|w|^{q-2}\,w+\Delta w'+f\qquad(p,\,q>1,\,p\neq2),
\end{equation}
is obtained choosing $\mathcal{W}$ as in \eqref{eq-plap} and $\mathcal{G}$ as in \eqref{eq-strongdamp}. In both cases, Assumptions \ref{ass-W} and \ref{ass-G} are satisfied (with $\mathrm{G}=H^1$ and either $\mathrm{W}=H^1\cap L^p$, $\theta=1-1/\max\{2,p\}$ in the first case, or $\mathrm{W}=\{v\in L^2:\nabla v\in L^p,\,v\in L^q\}$, $\theta=1-1/\max\{p,q\}$ in the second case). As a consequence, in both cases, by (a)--(d) of
Theorem \ref{result} there exists a sequence of minimizers of $F_\ep$ that converges to a function $w$ which fulfills \eqref{eq-ic} and \eqref{eq-enineqd}. However, while for the first equation
\eqref{eq-stronguno}
the functional $\mathcal{W}$ in \eqref{eq-strongdamp}
takes the form prescribed by \eqref{eq-Wass}, and thus $w$ is a solution of \eqref{eq-stronguno} by (e) of Theorem \ref{result}, in the case
of equation \eqref{eq-strongdue} item (e) of Theorem \ref{result}
does not apply, and it is an open problem
to establish whether $w$ actually solves \eqref{eq-strongdue}.

\bigskip
\noindent\textbf{4. Other damped wave equations.} These equations are
 obtained adding, for example, the sum of the above mentioned terms, i.e. $\Delta w'$ and $-w'$, with
 \[
 \mathcal{G}(v)=\frac{1}{2}\int_{\R^n}\left(|\nabla v|^2+|v|^2\right)\dx\qquad\text{and}\quad\mathrm{G}=H^1.
 \]
 On the other hand, higher order dissipative terms are allowed too, such as for instance $-\Delta^2 w'$, where
 \[
 \mathcal{G}(v)=\frac{1}{2}\int_{\R^n}|\Delta v|^2\dx\qquad\text{and}\quad\mathrm{G}=H^2,
\]
as well as the combinations with $\Delta w'$, $-w'$ and their sum (with an obvious definition of $\mathcal{G}$ and $\mathrm{G}$).


\section{Minimizers}

As we mentioned in Section \ref{sec-intro}, the first step of the De Giorgi approach is the search of the minimizers of \eqref{eq-fuc}. However, in order to drop the $\ep$ dependence in the weight of the integrals (and thus simplify the problem) it is convenient to minimize the auxiliary functional
\[
 J_{\ep}(u):=H_{\ep}(u)+\Delta_\ep(u)-S(u),
\]
where
\begin{gather}
 H_{\ep}(u):=\int_0^{\infty}e^{-t}\left(\frac{1}{2\ep^2}\int_{\R^n}|u''(t,x)|^2\dx+\mathcal{W}\big(u(t)\big)\right)\dt,\nonumber\\[.3cm]
 \Delta_{\ep}(u):=\frac{1}{\ep}\int_0^{\infty}e^{-t}\,\mathcal{G}(u'(t))\dt,\nonumber\\[.3cm]
 \label{eq-funzjsor}
 S(u):=\int_0^{\infty}\int_{\R^n}e^{-t}\,\phi(t,x)u(t,x)\dx\dt
\end{gather}
($\phi$ being a fixed function on $[0,\infty)\times\R^n$), which is equivalent to $F_{\ep}$ in the sense that, setting $\phi(t,x)=f_{\ep}(\ep t,x)$, there results $F_{\ep}(w)=\ep J_{\ep}(u)$, whenever $u(t,x)=w(\ep t,x)$. In other words, with a proper dilation of the boundary conditions (see e.g., \eqref{eq-bc}) the existence of minimizers $w_{\ep}$ for $F_{\ep}$ is equivalent to the existence of minimizers $u_{\ep}$ for $J_{\ep}$.

\medskip
In addition, following \cite{TT}, we define the weighted $L^2$-norm
\[
\Vert v\Vert_{\LL}^2:=\int_0^{\infty}\int_{\R^n}e^{-t}\,|v(t,x)|^2\dx\dt
\]
(with values in $[0,+\infty]$) and, for fixed $\ep$, we make the following assumptions on $\phi(t,x)$:
\begin{gather}
  \label{suppcomt} \text{there exists}\quad T^*\in\big(0,\ep^{-3/2}\big]\qquad\text{such that}\qquad\phi(t,x)=0,\quad\forall t>T^*,\\[.25cm]
 \label{eq-assogr} \|\phi\|_\LL \leq \ep,\\
 \label{assAqprel} \ep \int_0^t \int_s^\infty e^{-(\tau-s)}\,(\tau-s)\,\|\phi(\tau)\|_{L^2}^2\dtau\,ds\leq\gamma(\ep t+t_\ep)+\ep^2,\qquad\forall t\geq 0,
\end{gather}
where $t_\ep>0$ satisfies $\lim_{\ep\downarrow0} t_\ep=0$, while
\begin{equation}
\label{eq-gamma}
\gamma(t):=\int_0^t \Vert f(s)\Vert_{L^2}^2\,ds,\qquad t\geq 0
\end{equation}
(which is well defined since the forcing term $f$ of \eqref{eq-waved} belongs to $L_{loc}^2([0,\infty);L^2)$).

\medskip
In the sequel we shall use the following lemma (for a proof
see \cite[Lemma 2.3]{ST1}).
\begin{lemma}
\label{lem-poincare}
 If $u\in H^2_{\text{loc}}([0,\infty);L^2)$, then
 \[
  \Vert u'\Vert_\LL^2\leq 2\,\Vert u'(0)\Vert_{L^2}^2+4\,\Vert u''\Vert_\LL^2
 \]
 and
 \[
\Vert u\Vert_\LL^2\leq 2\,\Vert u(0)\Vert_{L^2}^2 +8\,\Vert u'(0)\Vert_{L^2}^2+16\,\Vert u''\Vert_\LL^2.
 \]
\end{lemma}
Now we are in a position to prove the existence of minimizers for $J_\ep$.
\begin{proposition}
 \label{prop-mind}
 Let $w_0\in\mathrm{W},\,w_1\in\mathrm{W}\cap\mathrm{G}$ (with $\mathrm{W}$ and $\mathrm{G}$ defined by \eqref{eq-domW} and \eqref{eq-defG}, respectively) and $\ep\in(0,1)$. Then, under Assumptions \ref{ass-W} and \ref{ass-G}, $J_{\ep}$ admits a minimizer $u_{\ep}$ in the class of functions $u\in H_{loc}^2([0,\infty);L^2)$ subject to the boundary conditions
 \begin{equation}
  \label{eq-bc}
  u(0)=w_0,\qquad u'(0)=\ep w_1.
 \end{equation}
 In addition,
 \begin{equation}
  \label{eq-levd}
  H_{\ep}(u_{\ep})+\Delta_{\ep}(u_{\ep})\leq \mathcal{W}(w_0)+\ep C.
 \end{equation}
\end{proposition}

\begin{proof}
Let $M$ be the set of functions in  $H^2_{\text{loc}}([0,\infty);L^2)$ satisfying \eqref{eq-bc}. If $u\in M$, then $S(u)$ is finite by \eqref{suppcomt}, and thus $J_\ep(u)$ is well defined (possibly equal to $+\infty$). In addition, whenever $J_\ep(u)$ is finite, arguing as in \cite[proof of Proposition 4.1]{TT} (and recalling that $\mathcal{G}\geq0$, as well as $\mathcal{W}$), there results
\[
 |S(u)|\leq\|\phi\|_\LL\|u\|_\LL\leq\ep C\left(1+\|u''\|_\LL\right)
\]
with $C$ depending on \eqref{eq-bc}. As a consequence, plugging into the definition of $J_\ep$,
\begin{equation}
\label{coerc}
 J_{\ep}(u)\geq\frac{1}{2\ep^2}\Vert u''\Vert_\LL^2+\int_0^\infty e^{-t}\big[\mathcal{W}\bigl(u(t)\bigr)+\tfrac{1}{\ep}\mathcal{G}\big(u'(t)\big)\big]\dt-\ep C\left(1+\|u''\|_\LL\right)
\end{equation}
and thus, using Lemma \ref{lem-poincare}, $J_{\ep}$ is coercive in $M$ with respect to the topology of $H_{\text{loc}}^2([0,\infty);L^2)$. Hence, every minimizing sequence has a subsequence weakly convergent in $H_{\text{loc}}^2([0,\infty);L^2)$, that preserves \eqref{eq-bc}. Finally, by the weak semicontinuity of $H_\ep$ and $\Delta_\ep$ (consequences of Assumptions \ref{ass-W} and \ref{ass-G}) and the weak continuity of $S(u)$, one finds a minimizer $u_\ep\in M$ of $J_\ep$.

Now, letting $\psi(t,x):=w_0(x)+\ep t w_1(x)$, and observing that $\psi\in M$ and $\psi''\equiv 0$, following \cite[proof of Lemma 3.1]{ST2} and \cite[proof of Proposition 4.2]{TT}, one obtains that
 \[
 H_\ep(\psi)+\Delta_\ep(\psi)\leq\mathcal{W}(w_0)+C\ep,
 \]
 and that
 \[
  -S(\psi)\leq\left(\|w_0\|_{L^2}+\sqrt{2}\,\ep\|w_1\|_{L^2}\right)\|\phi\|_\LL\leq C\ep.
 \]
 Thus, $J_\ep(\psi)\leq \mathcal{W}(w_0)+C\ep$, whence
 \begin{equation}
  \label{eq-est_alto}
  J_\ep(u_\ep)\leq \mathcal{W}(w_0)+C\ep.
 \end{equation}
 Moreover, combining $J_\ep(u_\ep)\leq C$ with \eqref{coerc} and Young's inequality, one finds that $\|u_\ep''\|_\LL \leq \ep C$, which entails (plugging into \eqref{coerc})
\begin{equation*}
   J_{\ep}(u_\ep)\geq
\frac{1}{2\ep^2}\Vert u''_\ep\Vert_\LL^2+\int_0^\infty e^{-t}\big[\mathcal{W}\bigl(u(t)\bigr)+\tfrac{1}{\ep}\mathcal{G}\big(u'(t)\big)\big]\dt-\ep C.
\end{equation*}
Finally, recalling \eqref{eq-est_alto}, one easily obtains \eqref{eq-levd}.
\end{proof}

Some notation is required in the sequel. First, we assume throughout that $\ep\in(0,1)$, as in Proposition \ref{prop-mind}. Furthermore, given a minimizer $u_{\ep}$ of $J_{\ep}$, we define for every $t\geq 0$:
\begin{gather}
 \label{eq-W}
 \mathcal{W}_{\ep}(t):=\mathcal{W}\big(u_{\ep}(t)\big),\\[.2cm]
 D_{\ep}(t):=\displaystyle\frac{1}{2\ep^2}\|u_{\ep}''(t)\|_{L^2}^2,\nonumber\\[.2cm]
 \label{eq-G}
 \mathcal{G}_{\ep}(t):=\mathcal{G}\big(u_{\ep}'(t)\big),\\[.3cm]
 \label{eq-L}
 L_{\ep}(t):=D_{\ep}(t)+\mathcal{W}_{\ep}(t)+\tfrac{1}{\ep}\mathcal{G}_{\ep}(t),\\[.3cm]
 \Phi_{\ep}(t):=\big(\phi(t),u_{\ep}'(t)\big)_{L^2},\nonumber\\[.3cm]
 \label{eq-kinen}
 K_{\ep}(t):=\frac{1}{2\ep^2}\|u_{\ep}'(t)\|_{L^2}^2.
\end{gather}
Note that $D_\ep$ and $L_\ep$ are actually defined almost everywhere.

\begin{proposition}
 \label{prop-jderd}
 Let $w_0,\,w_1\text{ and }\mathcal{W}$ satisfy the assumptions of Proposition \ref{prop-mind} and let $u_{\ep}$ be a minimizer of $J_{\ep}$. Then, for every $g\in C^2([0,\infty))$ constant for large $t$ and with $g(0)=0$,
 \begin{equation}
  \label{eq-relder}
  \begin{array}{l}
   \displaystyle\int_0^{\infty}e^{-t}\,\big(g'(t)-g(t)\big)L_{\ep}(t)\dt+\int_0^{\infty}e^{-t}\,g(t)\Phi_{\ep}(t)\dt+\\[.5cm]
   \hspace{3.5cm}\displaystyle-\int_0^{\infty}e^{-t}\,\big(4g'(t)D_{\ep}(t)+g''(t)K_{\ep}'(t)+\tfrac{2}{\ep}g'(t)\mathcal{G}_\ep(t)\big)\dt=g'(0)R(u_{\ep}),
  \end{array}
 \end{equation}
 where the residual term
 \begin{equation}
  \label{eq-resto}
  R(u_{\ep}):=\ep\int_0^{\infty}e^{-t}\big[\,t\,\big(-\left\langle\nabla\mathcal{W}\big(u_{\ep}(t)\big),w_1\right\rangle+(\phi(t),w_1)_{L^2}\big)-\tfrac{1}{\ep}\left\langle\nabla\mathcal{G}\big(u_{\ep}'(t)\big),w_1\right\rangle\big]\dt
 \end{equation}
 satisfies the estimate
  \begin{equation}
  \label{eq-restest}
  |R(u_{\ep})|\leq C\sqrt{\ep}.
   \end{equation}
\end{proposition}

\begin{proof}
The proof consists mainly in a combination between \cite[proof of Proposition 4.4]{ST2} and \cite[proof of Proposition 4.4]{TT}. For small $\delta$, let $\varphi_{\delta}(t):=t-\delta g(t)$ and $U_{\delta}(t):=u_{\ep}\big(\varphi_{\delta}(t)\big)+t\ep\delta g'(0)w_1$, which is an admissible competitor of $u_\ep$ in the minimization of $J_\ep$, since it satisfies \eqref{eq-bc}. Moreover, as $u_\ep$ is a minimizer and $U_{\delta=0}=u_\ep$,
\[
  \left.\frac{\partial}{\partial\delta}J_{\ep}(U_{\delta})\right|_{\delta=0}=\left.\frac{\partial}{\partial\delta}H_{\ep}(U_{\delta})\right|_{\delta=0}+\left.\frac{\partial}{\partial\delta}\Delta_{\ep}(U_{\delta})\right|_{\delta=0}-\left.\frac{\partial}{\partial\delta}S(U_{\delta})\right|_{\delta=0}=0.
\]
The first two derivatives (already computed in \cite[proof of Proposition 4.4]{ST2}) are given by
\begin{align*}
 \left.\frac{\partial}{\partial\delta}H_{\ep}(U_{\delta})\right|_{\delta=0}= & \, \int_0^\infty e^{-t}\,\left[\big(g'(t)-g(t)\big)\big(L_\ep(t)-\tfrac{1}{\ep}\mathcal{G}_\ep(t)\big)-4g'(t)D_\ep(t)-g''(t)K_\ep'(t)\right]\dt+\\[.2cm]
                                                                             & \, + \ep g'(0)\int_0^\infty e^{-t}\,t\,\left\langle\nabla\mathcal{W}\big(u_\ep(t)\big),w_1\right\rangle\dt
\end{align*}
and
\[
 \left.\frac{\partial}{\partial\delta}\Delta_{\ep}(U_{\delta})\right|_{\delta=0}=\int_0^\infty\frac{e^{t/\ep}}{\ep}\,\left[-\big(g'(t)+g(t)\big)\mathcal{G}_\ep(t)+\ep g'(0)\left\langle\nabla\mathcal{G}\big(u_\ep'(t)\big),w_1\right\rangle\right]\dt,
\]
while the last one was computed in \cite[proof of Proposition 4.4]{TT} and reads
\[
 \left.\frac{\partial}{\partial\delta}S(U_{\delta})\right|_{\delta=0}=\int_0^\infty e^{-t}\left[-g(t)\Phi_\ep(t)+\ep g'(0)t\big(\phi(t),w_1\big)_{L^2}\right]\dt.
\]
 Combining these results and suitably rearranging terms, one obtains \eqref{eq-relder} and \eqref{eq-resto}.

It remains to prove the estimate for $R(u_\ep)$. Arguing again as in \cite[proof of Proposition 4.4]{TT} one has
 \[
 \left|\int_0^{\infty}e^{-t}\,t\,\left\langle\nabla\mathcal{W}\big(u_{\ep}(t)\big),w_1\right\rangle\dt\right| \leq C(1+\ep),
 \]
 and
 \[
  \left|\int_0^{\infty}e^{-t}\,t\,\big(\phi(t),w_1\big)_{L^2}\dt\right|\leq C\ep,
 \]
 whereas, on the other hand, using the definition of $\mathcal{G}$, \eqref{eq-assogr} and \eqref{eq-levd},
 \[
  \left|\int_0^\infty e^{-t}\,\langle\nabla\mathcal{G}\big(u_\ep'(t)\big),w_1\rangle\dt\right|\leq C\int_0^\infty e^{-t}\,\sqrt{\mathcal{G}_\ep(t)}\dt\leq C\big(\ep\Delta_\ep(u_\ep)\big)^{1/2} \leq  C\sqrt{\ep}(1+\ep).
 \]
 Summing up, \eqref{eq-restest} is satisfied and this completes the proof.
\end{proof}

\begin{remark}
 Estimate \eqref{eq-restest} is weaker than the analogous for the nondissipative case (see \cite[Eq. (40)]{TT}) by a factor $\sqrt{\ep}$. However, we will prove that this does not affect the variational approach.
\end{remark}

Before showing the consequences of Proposition \ref{prop-jderd}, we have to recall the definition of the average operator (for details, we refer the reader to \cite[Section 3]{TT}).

\begin{definition}
 The \emph{average operator} is the operator that associates any measurable function $h:[0,\infty]\to[0,\infty]$ with the function $\mathcal{A}h$, defined by
 \begin{equation}\label{defA}
  \mathcal{A}h\,(t):=\int_t^{\infty}e^{-(s-t)}\,h(s)\ds,\qquad t\geq0.
 \end{equation}
\end{definition}

Note that, by an application of Fubini's Theorem,
 the iterated operator $\mathcal{A}^2$ can be represented as
\begin{equation}
\label{Aquadro}
 \mathcal{A}^2h\,(t):=\mathcal{A}(\mathcal{A}h)\,(t)=\int_t^{\infty}e^{-(s-t)}\,(s-t)\,h(s)\ds.
\end{equation}
Moreover,
$\mathcal{A}h\,(t)$ is well defined (and finite for every $t\geq 0$) even when $h$ is a changing sign function, provided that the quantity $\mathcal{A}|h|\,(0)$ is finite.

In addition, by virtue of \eqref{Aquadro}, 
condition \eqref{assAqprel} can be rewritten as 
\begin{equation}
 \label{assAq}
 \ep \int_0^t \mathcal{A}^2\|\phi(\cdot)\|_{L^2}^2\,(s)\,ds\leq\gamma(\ep t+t_\ep)+\ep^2,\qquad\forall t\geq 0.
\end{equation}

\begin{corollary}
Under the assumptions of Proposition \ref{prop-jderd}, there results
 \begin{equation}
  \label{eq-relzero}
  \mathcal{A}^2L_{\ep}\,(0)+4\mathcal{A}D_{\ep}\,(0)+\tfrac{2}{\ep}\mathcal{A}\mathcal{G}_\ep\,(0)-\mathcal{A}L_{\ep}\,(0)=\mathcal{A}^2\Phi_{\ep}\,(0)-R(u_{\ep})
 \end{equation}
 and
 \begin{equation}
  \label{eq-relt}
  \mathcal{A}^2L_{\ep}\,(t)+4\mathcal{A}D_{\ep}\,(t)+\tfrac{2}{\ep}\mathcal{A}\mathcal{G}_\ep\,(t)-\mathcal{A}L_{\ep}\,(t)=\mathcal{A}^2\Phi_{\ep}\,(t)-K_{\ep}'(t),\qquad\mbox{for a.e. } t>0.
 \end{equation}
\end{corollary}

\begin{proof}
Formally, \eqref{eq-relzero} is obtained choosing $g(t)=t$ in \eqref{eq-relder} (and rewriting the resulting integrals
with the aid of \eqref{defA} and \eqref{Aquadro}). However, since the function $g(t)=t$ does not satisfy the assumptions of Proposition \ref{prop-jderd}, it is necessary to approximate it from below by a suitable sequence $(g_k)$ fulfilling the requirements of Proposition \ref{prop-jderd} and such that $g_k'(0)=1$, and pass to the limit in \eqref{eq-relder}. This can be done arguing exactly as in \cite[proof of Corollary 4.5]{ST2} and managing the term involving $\Phi_\ep$ as suggested by \cite[proof of Corollary 4.5]{TT} (namely, using \eqref{suppcomt}, \eqref{eq-assogr} and dominated convergence).

Concerning \eqref{eq-relt}, the proof follows exactly \cite[proof of Corollary 4.7]{ST2} (namely, one lets $g(s)=(s-t)^+$ in \eqref{eq-relder}, 
which can be justified by an approximation argument, treating again the term with $\Phi_\ep$ as in \cite[proof of Corollary 4.5]{TT}).
\end{proof}


\section{Energy estimates}

Following \cite{ST2,TT}, the next step is the study of the approximate energy. However, as we pointed out in Section \ref{sec-intro}, here we actually investigate a variant of this quantity which takes into account the effect of the dissipation.

Recall that the approximate energy associated with a minimizer $u_{\ep}$ of $J_{\ep}$ (obtained by Proposition \ref{prop-mind}) is given by
\[
  E_{\ep}(t):=\frac{1}{2\ep^2}\int_{\R^n}|u_{\ep}'(t,x)|^2\dx+\int_t^{\infty}e^{-(s-t)}(s-t)\,\mathcal{W}\big(u_{\ep}(s)\big)\ds.
 \]
 Then, we define its dissipative correction as follows.
\begin{definition}
 The \emph{dissipative (correction of the) approximate energy} is defined as
 \begin{equation}
  \label{eq-dappenesp}
  E_{\ep}^d(t):=E_{\ep}(t)+\frac{1}{\ep}\int_0^t\int_s^{\infty}e^{-(\tau-s)}(1+\tau-s)\,\mathcal{G}(u_\ep'(\tau))\dtau\ds.
 \end{equation}
 \qed
\end{definition}

\begin{remark}
 Note that, by \eqref{eq-W}, \eqref{eq-G} and \eqref{eq-kinen}, \eqref{eq-dappenesp} reads
 \begin{equation}
  \label{eq-dappenesp_comp}
  E_{\ep}^d(t)=\underbrace{K_{\ep}(t)+\mathcal{A}^2\mathcal{W}_{\ep}\,(t)}_{:=E_\ep(t)}+\frac{1}{\ep}\int_0^t\big(\mathcal{A}\mathcal{G}_\ep\,(s)+\mathcal{A}^2\mathcal{G}_\ep\,(s)\big)\ds.
 \end{equation}
 In addition, we observe that $E_\ep^d(0)=E_\ep(0)$ and that, in view of \eqref{defead}, $\mathcal{E}_\ep^d(t)=E_\ep^d(t/\ep)$.
\end{remark}

Now, exploiting the fact that $E_\ep^d\geq E_\ep$ and arguing as in \cite{TT}, we can establish an upper bound for the time evolution of the approximate energy. First, we estimate its value at $t=0$.

\begin{lemma}[Estimate for $E_{\ep}^d(0)$]
 We have that
 \begin{equation}
  \label{eq-appenzero}
  \Lambda_\ep:=E_{\ep}^d(0)+\frac{1}{\ep}\big(\mathcal{A}^2\mathcal{G}_\ep\,(0)+2\mathcal{A}\mathcal{G}_\ep\,(0)\big)\leq\frac{1}{2}\|w_1\|_{L^2}^2+\mathcal{W}(w_0)+C\sqrt{\ep}.
 \end{equation}
\end{lemma}

\begin{proof}
 From \eqref{eq-bc}, $E_{\ep}^d(0)=E_{\ep}(0)=\frac{1}{2}\|w_1\|_{L^2}^2+\mathcal{A}^2\mathcal{W}_{\ep}\,(0)$. Since $\mathcal{A}^2\mathcal{W}_{\ep}\,(0)\leq \mathcal{A}^2L_{\ep}\,(0)-\frac{1}{\ep}\mathcal{A}^2\mathcal{G}_\ep\,(0)$, from \eqref{eq-relzero} we obtain
 \[
  \mathcal{A}^2\mathcal{W}_{\ep}\,(0)\leq\mathcal{A}^2\Phi_{\ep}\,(0)+\mathcal{A}L_{\ep}\,(0)-R(u_{\ep})-\frac{1}{\ep}\big(\mathcal{A}^2\mathcal{G}_\ep\,(0)+2\mathcal{A}\mathcal{G}_\ep\,(0)\big).
 \]
 Recalling that $\mathcal{A}L_{\ep}\,(0)=H_{\ep}(u_{\ep})$ and combining with \eqref{eq-levd} and \eqref{eq-restest},
 \begin{equation}
  \label{eq-a_parz}
  \mathcal{A}^2\mathcal{W}_{\ep}\,(0)\leq \mathcal{A}^2\Phi_{\ep}\,(0)-\frac{1}{\ep}\big(\mathcal{A}^2\mathcal{G}_\ep\,(0)+2\mathcal{A}\mathcal{G}_\ep\,(0)\big)+\mathcal{W}(w_0)+C\sqrt{\ep}.
 \end{equation}
 Moreover, arguing as in \cite[proof of Lemma 5.3]{TT}, 
 from \eqref{suppcomt} and\eqref{eq-assogr} we obtain
 \[
  |\mathcal{A}^2\Phi_{\ep}\,(0)|\leq C \frac{\|u_{\ep}'\|_\LL}{\sqrt{\ep}},
 \]
 while from Lemma \ref{lem-poincare}, \eqref{eq-bc} and \eqref{eq-levd} we
 obtain
 \[
  \|u_{\ep}'\|_\LL^2\leq C\ep^2.
 \]
 Hence $|\mathcal{A}^2\Phi_{\ep}\,(0)|\leq C\sqrt{\ep}$, so that, suitably rearranging terms in \eqref{eq-a_parz}, \eqref{eq-appenzero} is proved.
\end{proof}

Now we are ready to establish an estimate for $E_{\ep}^d(t/\ep)$. This will be obtained as a consequence of \eqref{suppcomt}, \eqref{eq-assogr} and \eqref{assAq}, using \cite[Lemma 5.6]{TT}, a Gr\"onwall-type
 lemma that claims the following:
 if
\begin{itemize}
 \item[1)] $c:[a,b]\to\R$ is a positive, differentiable and nondecreasing function;
 \item[2)] $u$ and $v$ are two nonnegative functions such that $u\in C^0([a,b])$ and $v\in L^1([a,b])$;
 \item[3)] $c$, $u$ and $v$ satisfy $u(t)\leq c^2(t)+2\int_a^tv(s)\sqrt{u(s)}\ds$, for every $t\in[a,b]$;
\end{itemize}
then
\begin{equation}
\label{eq-gronwall}
 \sqrt{ u(t)}\leq c(t)+\int_a^tv(s)\ds,\qquad\forall t\in[a,b].
 \end{equation}

\begin{proposition}[Dissipative approximate energy estimate]
 \label{prop-appest}
 For every $\beta>1$, there exists a constant $C_{\beta}>0$ such that for every $T\geq 0$
\begin{equation}
  \label{stimalocT}
   \sqrt{ E_{\ep}^d(T/\ep)}\leq \sqrt{E_{\ep}^d(0)}+\left(\sqrt{\ep C_\beta}+\sqrt{T \beta/2}\right)\sqrt{\gamma(T+t_\ep)+\ep^2},\qquad\forall \ep\in (0,1).
\end{equation}
In particular, for every $T\geq 0$ there exists $C_T$  such that
\begin{equation}
  \label{stimalocT2}
   E_{\ep}^d(t/\ep)\leq C_T,\qquad\forall \ep\in (0,1),\quad \forall t\in [0,T].
\end{equation}
\end{proposition}

\begin{proof}
 First, using the properties of $\mathcal{A}$ (precisely, \cite[Eq. (47)]{ST2}), \eqref{eq-L} and \eqref{eq-relt}, there results
 \[
  \frac{d}{dt}E_\ep^d(t)=-3\mathcal{A}D_{\ep}\,(t)-\mathcal{A}^2D_{\ep}\,(t)+\mathcal{A}^2\Phi_{\ep}\,(t),\qquad\mbox{for a.e. }t\geq0.
 \]
 Therefore, recalling the definitions of $\mathcal{A}^2$ and $\Phi_{\ep}$ and observing that $e^{-(s-t)}(s-t)$ is a probability kernel on $[t,\infty)$, we have
 \[
  \frac{d}{dt}E_\ep^d(t)\leq-3\mathcal{A}D_{\ep}\,(t)-\mathcal{A}^2D_{\ep}\,(t)+N_{\phi}(t)\left(\mathcal{A}^2\|u_{\ep}'(\cdot)\|_{L^2}^2\,(t)\right)^{1/2}
 \]
 where $N_{\phi}(t)=\left(\mathcal{A}^2\|\phi(\cdot)\|_{L^2}^2\,(t)\right)^{1/2}$. Now, using \cite[Lemma 3.4]{TT} and Young's inequality, one may check that
\[
  \frac{d}{dt}E_\ep^d(t)\leq\sqrt{2\beta}\,\ep N_{\phi}(t) \sqrt{E_\ep(t)}+C_{\beta}\ep^2 N_{\phi}(t)^2
\]
and, integrating on $[0,t]$,
 \[
  E_{\ep}^d(t)\leq E_{\ep}^d(0)+C_{\beta}\ep^2\int_0^tN_{\phi}^2(s)\ds
  +\sqrt{2\beta}\,\ep\int_0^t N_{\phi}(s)\sqrt{E_{\ep}(s)}\ds.
 \]
 As a consequence, if we set
 \[
  u(t)=E_{\ep}^d(t),\qquad v(t)=\ep \sqrt{\beta/2} N_{\phi}(t),\qquad
  c(t)^2=E_{\ep}(0)+C_{\beta}\,\ep^2\int_0^tN_{\phi}^2(s)\ds,
 \]
 then \eqref{eq-gronwall} yields
 \[
  \sqrt{E_{\ep}^d(t)}\leq \left(E_{\ep}^d(0)+C_{\beta}\ep^2\int_0^tN_{\phi}^2(s)\ds\right)^{1/2}
  +\ep \sqrt{\beta/2}\int_0^t N_{\phi}(s)\ds
 \]
 and thus, using Cauchy-Schwarz, with some computations one obtains
 \[
 \sqrt{ E_{\ep}^d(t)}\leq \sqrt{E_{\ep}^d(0)}+\sqrt{\ep}\left(\sqrt{C_\beta}+\sqrt{t\beta/2}\right)\sqrt{\ep\int_0^t\mathcal{A}^2\|\phi(\cdot)\|_{L^2}^2\,(s)\,ds}.
 \]
 Hence \eqref{stimalocT} follows exploiting \eqref{assAq} and setting $t=T/\ep$. Finally, one easily gets \eqref{stimalocT2}, observing that the right hand side of \eqref{stimalocT} is increasing with respect to $T$, that $t_\ep\downarrow0$, that $\beta$ can be fixed (e.g. $\beta=2$) and that
$E_\ep^d(0)\leq C$ by \eqref{eq-appenzero}.
\end{proof}

Proposition \ref{prop-appest} has an immediate consequence, which is crucial for proving \eqref{eq-aprdued}, \eqref{eq-aprtred} and \eqref{eq-enineqd}.

\begin{corollary}
 \label{corl-appest}
 For every $\beta>1$, there exists a constant $C_{\beta}>0$ such that for every $T\geq 0$
\begin{equation}
  \label{stimalocTbis}
   E_{\ep}(T/\ep)+\frac{2}{\ep}\int_0^{T/\ep}\mathcal{G}_\ep(t)\dt\leq \left(\sqrt{\Lambda_\ep}+\left(\sqrt{\ep C_\beta}+\sqrt{T \beta/2}\right)\sqrt{\gamma(T+t_\ep)+\ep^2}\right)^2,\qquad\forall \ep\in (0,1).
\end{equation}
In particular, for every $T\geq 0$ there exists $C_T$  such that
\begin{equation}
  \label{stimalocT2bis}
   E_{\ep}(t/\ep)+\frac{2}{\ep}\int_0^{t/\ep}\mathcal{G}_\ep(s)\ds\leq C_T,\qquad\forall \ep\in (0,1),\quad \forall t\in [0,T].
\end{equation}
\end{corollary}

\begin{proof}
 From \cite[Lemma 3.2]{TT} one has
 \[
  \frac{1}{\ep}\int_0^t\big(\mathcal{A}\mathcal{G}_\ep\,(s)+\mathcal{A}^2\mathcal{G}_\ep\,(s)\big)\ds=\frac{2}{\ep}\int_0^t\mathcal{G}_\ep(s)\ds+\frac{1}{\ep}\big(\mathcal{A}^2\mathcal{G}_\ep\,(t)-\mathcal{A}^2\mathcal{G}_\ep\,(0)+2\mathcal{A}\mathcal{G}_\ep\,(t)-2\mathcal{A}\mathcal{G}_\ep\,(0)\big),
 \]
 so that
 \[
   E_{\ep}(t)+\frac{2}{\ep}\int_0^t\mathcal{G}_\ep(s)\ds\leq E_\ep^d(t)+\frac{1}{\ep}\big(\mathcal{A}^2\mathcal{G}_\ep\,(0)+2\mathcal{A}\mathcal{G}_\ep\,(0)\big).
 \]
 Hence, combining the previous inequality with \eqref{stimalocT} and recalling that $\Lambda_\ep\geq E_\ep^d(0)$, one easily obtains \eqref{stimalocTbis}. Finally, \eqref{stimalocT2bis} follows arguing as in the last part of the proof of Proposition \ref{prop-appest}.
\end{proof}

\begin{remark}
 It is straightforward that also $E_{\ep}(t/\ep)+\frac{2}{\ep}\int_0^{t/\ep}\mathcal{G}_\ep(s)\ds$ is a formal approximation of the dissipative correction of the mechanical energy $\mathcal{E}^d$, as well as $E_\ep^d(t/\ep)$.
\end{remark}


\section{Proof of Theorem \ref{result}}

Before showing the proof we recall some facts. From \cite[Lemma 6.1]{TT}, for every function $f\in L_{\text{loc}}^2([0,\infty);L^2)$ there exists a sequence $(f_{\ep})\subset L_{\text{loc}}^2([0,\infty);L^2)$ satisfying the following properties:
 \begin{itemize}
  \item[(i)] as $\ep\downarrow0$, $f_{\ep}\to f$ in $L^2([0,T];L^2)$ and $\|f_{\ep}\|_{L^2([0,T];L^2)}\uparrow\|f\|_{L^2([0,T];L^2)}$, for every $T>0$;
  \item[(ii)] $\mathrm{supp}\{f_{\ep}\}\subset[t_{\ep},T_{\ep}]\times\R^n$, with $t_{\ep}>0$ and $T_{\ep}<\infty$;
  \item[(iii)] as $\ep\downarrow0$, $t_{\ep}\downarrow0$ and $T_{\ep}\uparrow\infty$, and moreover $\ep T_{\ep}\leq \sqrt{\ep}$ and $e^{-t_{\ep}/\ep}\left(1+\frac{T_{\ep}}{\ep}\right)\leq \ep^3$;
  \item[(iv)] for every $\ep\in(0,1)$, $ \quad\displaystyle   \int_{t_{\ep}}^{T_{\ep}}\|f_{\ep}(t)\|_{L^2}^2\dt\leq 1/\ep\quad$ and $ \quad\displaystyle\int_0^{\infty}e^{-t}\,\|f_{\ep}(\ep t)\|_{L^2}^2\dt \leq \ep^3$.
 \end{itemize}
In addition, from \cite[Corollary 6.2]{TT}, if $f\in L_{loc}^2([0,\infty);L^2)$ and $(f_\ep)$ is a sequence satisfying (i)--(iv), then for every fixed $\ep\in(0,1)$ the function
\begin{equation}
 \label{eq-assfi}
 \phi(t,x):=f_{\ep}(\ep t,x),\qquad t\geq0,\quad x\in\R^n,
\end{equation}
satisfies \eqref{suppcomt}, \eqref{eq-assogr} and \eqref{assAq}.

Now, we have all the ingredients to prove Theorem \ref{result}.

\begin{proof}[Proof of Theorem \ref{result}: item (a)]
 If we let $(f_\ep)$ be a sequence satisfying (i)--(iv) and set \eqref{eq-assfi} in \eqref{eq-funzjsor}, then assumptions \eqref{suppcomt}, \eqref{eq-assogr} and \eqref{assAq} and all the hypothesis of Proposition \ref{prop-mind} are fulfilled, so that we have a minimizer $u_{\ep}$, in the class of functions $u\in H_{\text{loc}}^2([0,\infty);L^2)$ subject to \eqref{eq-bc}, that satisfies \eqref{eq-levd}. Since
  \[
   F_{\ep}(w)=\ep J_{\ep}(u)\qquad\mbox{whenever}\qquad u(t,x)=w(\ep t,x),
  \]
  if we set
  \begin{equation}
  \label{eq-relmin}
  w_{\ep}(t,x):=u_{\ep}(t/\ep,x),\qquad t\geq0,\quad x\in\R^n,
 \end{equation}
  then $w_{\ep}$  is the required minimizer.
\end{proof}

\begin{proof}[Proof of Theorem \ref{result}: item (b)]
 We split the proof in three steps.

 \emph{Step (i): Euler-Lagrange equation of $u_\ep$.} Let $\eta(t,x):=\varphi(t)h(x)$, where $h\in\mathrm{W}\cap\mathrm{G}$, and $\varphi\in C^{1,1}([0,\infty))$ and satisfies $\varphi(0)=\varphi'(0)=0$. If we define $g(\delta):=J_{\ep}(u_{\ep}+\delta\eta)$ and observe that $g'(0)=0$, then we obtain
 \begin{align}
  \label{eq-elue}
  \frac{1}{\ep^2}\int_0^{\infty}e^{-t}\,\big(u_{\ep}''(t),\eta''(t)\big)_{L^2}\dt= & \, \int_0^{\infty}e^{-t}\,\left(-\left\langle\nabla\mathcal{W}\big(u_{\ep}(t)\big),\eta(t)\right\rangle+\big(f_{\ep}(\ep t),\eta(t)\big)_{L^2}\right)\dt+\nonumber\\[.3cm]
                                                                                   & \, -\frac{1}{\ep}\int_0^\infty e^{-t}\left\langle\nabla\mathcal{G}\big(u_{\ep}'(t)\big),\eta'(t)\right\rangle\dt
 \end{align}
 (exploiting the same arguments of \cite[Lemma 6.4]{TT} for the non dissipative terms and dominated convergence for the dissipative term). Moreover, the same relation holds if $\eta\in C_0^{\infty}(\R^+\times\R^n)$ by a classical density argument (see \cite[proof of Lemma 5.1]{ST2} for more details).

 \emph{Step (ii): representation formula for $u_{\ep}''$.} As in \cite[proof of Eqs. (2.11)$\&$(2.16)]{ST2} and \cite[proof of Lemma 6.5]{TT}, for every $\tau>0$ fixed, one can check that setting $\eta_k(t,x)=h(x)\phi_k(t)$, with $h\in\mathrm{W}\cap\mathrm{G}$ and $(\phi_k)$ a sequence of $C^{1,1}$ functions suitably approximating $(t-\tau)^+$, then \eqref{eq-elue} holds for every $\eta_k$ and passes to the limit (at every Lebesgue point $\tau$ of $\big(u_{\ep}''(\tau),h\big)_{L^2}$) yielding
 \begin{equation}
  \label{eq-rep}
  \frac{1}{\ep^2}\big(u_{\ep}''(\tau),h\big)_{L^2}=-\mathcal{A}^2\omega_1\,(\tau)+\mathcal{A}^2\omega_2\,(\tau)-\tfrac{1}{\ep}\mathcal{A}\omega_3\,(\tau),\qquad\mbox{for a.e.}\quad \tau>0,
 \end{equation}
 where $\omega_1(\tau):=\left\langle\nabla\mathcal{W}\big(u_{\ep}(\tau)\big),h\right\rangle$, $\omega_2(\tau):=\big(f_{\ep}(\ep \tau),h\big)_{L^2}$ and $\omega_3(\tau):=\left\langle\nabla\mathcal{G}\big(u_{\ep}'(\tau)\big),h\right\rangle$.

 \emph{Step (iii): proof of \eqref{eq-apruno}--\eqref{eq-aprtred}.} The proof of \eqref{eq-apruno} and \eqref{eq-aprdue} is simply the repetition of the arguments developed in the first part of \cite[Proof of Theorem 2.3: part (b)]{TT}, with $E_\ep(t)$ replaced by $E_\ep^d(t)$. Concerning \eqref{eq-aprdued}, this is an immediate consequence of \eqref{stimalocT2bis} in view of \eqref{eq-relmin} and \eqref{eq-defG}. Hence, it is left to show \eqref{eq-aprtred}. By \eqref{eq-rep}, $u_{\ep}''/\ep^2$ can be written as $\Psi_1+\Psi_2$, with
\[
 (\Psi_1(t),h)_{L^2}=-\mathcal{A}^2\omega_1\,(t)+\mathcal{A}^2\omega_2\,(t),\qquad\mbox{for a.e. }t>0,\quad\forall h\in\mathrm{W}\cap\mathrm{G},
\]
and
\[
 (\Psi_2(t),h)_{L^2}=-\tfrac{1}{\ep}\mathcal{A}\omega_3\,(t),\qquad\mbox{for a.e. }t>0,\quad\forall h\in\mathrm{W}\cap\mathrm{G}.
\]
Arguing as in the second part of \cite[Proof of Theorem 2.3: part (b)]{TT}, we find that for every $T>0$
\[
 \int_0^T\|\Psi_1(t/\ep)\|_{\mathrm{W}'}^2\dt\leq C_T.
\]
On the other hand, by \cite[Eq. (33)]{ST2} (note that in \cite{ST2} $\mathcal{G}$ and $\mathrm{G}$ are denoted by $\mathcal{H}$ and $\mathrm{H}$, respectively),
\[
 |\omega_3(t)|\leq C\|h\|_{\mathrm{G}}\sqrt{\mathcal{G}_{\ep}(t)}
\]
and then, with easy computations,
\[
 \|\Psi_2(t/\ep)\|_{\mathrm{G}'}\leq\frac{C}{\ep}\left(\mathcal{A}\mathcal{G}_{\ep}\,(t/\ep)\right)^{1/2}.
\]
Squaring the inequality and integrating over $[0,T]$, from \eqref{eq-dappenesp_comp} and \eqref{stimalocT2} we find that
\[
 \int_0^T\|\Psi_2(t/\ep)\|_{\mathrm{G}'}^2\dt\leq\frac{C}{\ep}\int_0^{T/\ep}\mathcal{A}\mathcal{G}_{\ep}\,(t)\dt\leq  C_T.
\]
Consequently, observing that $w_{\ep}''(t)=\Psi_1(t/\ep)+\Psi_2(t/\ep)$, we obtain \eqref{eq-aprtred}.
\end{proof}

Before showing the remaining part of the proof, we point out that, in the sequel, we deal with a sequence of minimizers $w_{\ep_i}$ and we will tacitly extract several subsequences. However, for ease of notation, we denote by $w_{\ep}$ the original sequence, as well as the subsequences we extract. The same holds for all the other quantities depending on $\ep$.

\begin{proof}[Proof of Theorem \ref{result}: item (c)]
 First we note that by \eqref{eq-apruno} and \eqref{eq-aprtred}
 \[
  \|w_{\ep}\|_{H^1([0,T];L^2)}\leq C_T,\qquad \|w_{\ep}'\|_{L^{\infty}([0,T];L^2)}\leq C_T,\qquad \|w_{\ep}'\|_{H^1([0,T];(\mathrm{W}\cap\mathrm{G})')}\leq C_T
 \]
 and, arguing as in \cite[proof of Theorem 2.5]{ST2}, one obtains convergence in $H^1([0,T];L^2)$, \eqref{eq-ic} (with the latter meant as an equality in $(\mathrm{W}\cap\mathrm{G})'$) and \eqref{eq-reg}. On the other hand, \eqref{eq-Gnor}, \eqref{eq-apruno} and \eqref{eq-aprdued} imply
 \[
  \int_0^T\|w_\ep'(t)\|_{\mathrm{G}}^2\dt\leq C_T,
 \]
 which proves \eqref{eq-regpd}, by a standard Banach-Alaoglu argument.
\end{proof}

\begin{proof}[Proof of Theorem \ref{result}: item (d)]
 From Corollary \ref{corl-appest}, for every $\beta>1$, there exists $C_\beta>0$ such that: letting $l(t):=\mathcal{W}_{\ep}(t)$ and
 \[
  m(t):=-K_{\ep}(t)-\frac{2}{\ep}\int_0^t+\left(\sqrt{\Lambda_\ep}+\big(\sqrt{\ep C_\beta}+\sqrt{\ep t\beta/2}\big)\sqrt{\gamma(\ep t+t_\ep)+\ep^2}\right)^2
 \]
 the assumptions of \cite[Lemma 6.1]{ST2} are satisfied, so that for every $a>0$, $\delta\in(0,1)$ and $T>0$,
 \begin{multline*}
  Y(\delta a)\int_{T+\delta a}^{T+a}\mathcal{W}_{\ep}(t)\dt+\int_T^{T+a}K_{\ep}(t)\dt+\frac{2}{\ep}\int_T^{T+a}\int_0^t\mathcal{G}_\ep(s)\ds\dt\leq\\[.2cm]
  \leq\int_T^{T+a}\left(\sqrt{\Lambda_\ep}+\big(\sqrt{\ep C_\beta}+\sqrt{\ep t\beta/2}\big)\sqrt{\gamma(\ep t+t_\ep)+\ep^2}\right)^2\dt,
 \end{multline*}
 where $Y(z):=\int_0^ze^{-s}\,s\ds$. Replacing $a$ with $a/\ep$ and $T$ with $T/\ep$, and recalling \eqref{eq-defG}, with a change of variable the previous inequality reads
 \begin{multline*}
  Y\left(\frac{\delta a}{\ep}\right)\int_{T+\delta a}^{T+a}\mathcal{W}\big(w_{\ep}(t)\big)\dt+\frac{1}{2}\int_T^{T+a}\|w_{\ep}'(t)\|_{L^2}^2(t)\dt+2\int_T^{T+a}\,\int_0^t\mathcal{G}(w_{\ep}'(s))\ds\dt\\[.2cm]
  \leq\int_T^{T+a}\left(\sqrt{\Lambda_\ep}+\big(\sqrt{\ep C_\beta}+\sqrt{t\beta/2}\big)\sqrt{\gamma(t+t_\ep)+\ep^2}\right)^2\dt.
 \end{multline*}
 Furthermore, as $\ep\downarrow0$, by definition $Y\left(\frac{\delta a}{\ep}\right)\to1$, whereas by \eqref{eq-appenzero} and \eqref{eq-gamma}
 \[
  \int_T^{T+a}\left(\sqrt{\Lambda_\ep}+\left(\sqrt{\ep C_\beta}+\sqrt{t\beta/2}\right)\sqrt{\gamma(t+t_\ep)+\ep^2}\right)\dt\to\int_T^{T+a}\left(\sqrt{\mathcal{E}^d(0)}+\sqrt{t\gamma(t)\beta/2}\right)\dt.
 \]
 Hence, arguing as in \cite[proof of Theorem 2.4]{ST2},
 \begin{multline*}
  \int_{T+\delta a}^{T+a}\mathcal{W}\big(w(t)\big)\dt+\frac{1}{2}\int_T^{T+a}\|w'(t)\|_{L^2}^2(t)\dt+2\int_T^{T+a}\,\int_0^t\mathcal{G}(w'(s))\ds\dt\leq\\[.2cm]
  \leq \int_T^{T+a}\left(\sqrt{\mathcal{E}^d(0)}+\sqrt{t\gamma(t)\beta/2}\right)^2\dt
 \end{multline*}
 and, letting $\delta\downarrow0$ and, subsequently, dividing by $a$ and letting $a\downarrow0$,  we obtain
 \[
  \mathcal{W}\big(w(T)\big)+\frac{1}{2}\|w'(T)\|_{L^2}^2+2\int_0^T\mathcal{G}(w'(t))\dt\leq \left(\sqrt{\mathcal{E}^d(0)}+\sqrt{T\gamma(T)\beta/2}\right)^2,\qquad\mbox{for a.e.}\quad T\geq0.
 \]
 Since the inequality is valid for every $\beta>1$, letting $\beta\downarrow1$, \eqref{eq-enineqd} follows.
\end{proof}

\begin{proof}[Proof of Theorem \ref{result}: item (e)]
 Arguing as in \cite[proof of Lemma 6.2]{ST2}, from \eqref{eq-elue} one obtains that for every function $\varphi\in C_0^{\infty}(\R^+\times\R^n)$
 \begin{multline}
 \label{eq-elwe}
  \int_0^{\infty}(w_{\ep}'(t),\ep^2\,\varphi'''(t)+2\ep\varphi''(t)+\varphi'(t))_{L^2}\dt=\int_0^{\infty}\langle\nabla\mathcal{W}(w_{\ep}(t)),\varphi(t)\rangle\dt+\\[.2cm]
  +\int_0^{\infty}\langle\nabla\mathcal{G}(w_{\ep}'(t)),\varphi(t)+\ep\varphi'(t)\rangle\dt-\int_0^{\infty}(f_{\ep}(t),\varphi(t))_{L^2}\dt.
\end{multline}
On the other hand, clearly
\[
  \int_0^{\infty}\big(w_{\ep}'(t),\ep^2\,\varphi'''(t)+2\ep\varphi''(t)+\varphi'(t)\big)_{L^2}\dt\to\int_0^{\infty}\big(w'(t),\varphi'(t)\big)_{L^2}\dt,
 \]
 whereas, by construction,
 \[
  \int_0^{\infty}\big(f_{\ep}(t),\varphi(t)\big)_{L^2}\dt\to\int_0^{\infty}\big(f(t),\varphi(t)\big)_{L^2}\dt.
 \]
 On the other hand, by \eqref{eq-apruno}, \eqref{eq-aprdue}, \eqref{eq-Wass} and \cite[Theorem 5.1]{lions2} (arguing as in \cite{ST2}),
 \[
  \int_0^{\infty}\left\langle\nabla\mathcal{W}\big(w_{\ep}(t)\big),\varphi(t)\right\rangle\dt\to\int_0^{\infty}\left\langle\nabla\mathcal{W}\big(w(t)\big),\varphi(t)\right\rangle\dt,
 \]
 while, by \cite[Eq. (33)]{ST2} and \eqref{eq-regpd},
 \[
 \int_0^{\infty}\langle\nabla\mathcal{G}(w_{\ep}'(t)),\varphi(t)+\ep\varphi'(t)\rangle\dt\to\int_0^{\infty}\langle\nabla\mathcal{G}(w'(t)),\varphi(t)\rangle\dt.
\]
Summing up, as $\ep\downarrow0$, equation \eqref{eq-elwe} tends to \eqref{eq-wavedfweak}, which concludes the proof.
\end{proof}



\begin{thebibliography}{99}

\bibitem{akagi}
G. Akagi, U. Stefanelli,
Doubly nonlinear equations as convex minimization,
\emph{SIAM J. Math. Anal.} \textbf{46} (2014), no. 3, 1922--1945.

\bibitem{bogelein}
V. B\"ogelein, F. Duzaar, P. Marcellini,
Existence of evolutionary variational solutions via the calculus of variations,
\emph{J. Differential Equations} \textbf{256} (2014), no. 12, 3912--3942.

\bibitem{cherrier}
P. Cherrier, A. Milani,
\emph{Linear and quasi-linear evolution equations in Hilbert spaces},
Graduate Studies in Mathematics 135, AMS, Providence, RI, 2012.

\bibitem{degiorgi1}
E. De Giorgi,
Conjectures concerning some evolution problems. A celebration of John F. Nash, Jr. (Italian),
\emph{Duke Math. J.} \textbf{81} (1996), no. 2, 255--268.

\bibitem{degiorgi2}
E. De Giorgi,
\emph{Selected Papers},
Edited by L. Ambrosio, G. Dal Maso, M. Forti, M. Miranda and S. Spagnolo, Springer-Verlag, Berlin, 2006.

\bibitem{kalantarov}
V. Kalantarov, S. Zelik,
Finite-dimensional attractors for the quasi-linear strongly-damped wave equation,
\emph{J. Differential Equations} \textbf{247} (2009), no. 4, 1120--1155.

\bibitem{liero}
M. Liero, U. Stefanelli,
A new minimum principle for Lagrangian mechanics,
\emph{J. Nonlinear Sci.} \textbf{23} (2013), no. 2, 179-204.

\bibitem{liero2}
M. Liero, U. Stefanelli,
Weighted inertia-dissipation-energy functionals for semilinear equations,
\emph{Boll. Unione Mat. Ital. (9)} \textbf{6} (2013), no. 1, 1--27.

\bibitem{lions4}
J.-L. Lions,
Probl\`emes aux limites en th\'eorie des distributions (French),
\emph{Acta Math.} \textbf{94} (1955), 13--153.

\bibitem{lions2}
J.-L. Lions,
\emph{Quelques m\'ethodes de r\'esolution des probl\`emes aux limites non lin\'eaires} (French),
Dunod; Gauthier-Villars, Paris, 1969.

\bibitem{melchionna}
S. Melchionna,
A variational principle for nonpotential perturbations of gradient flows of nonconvex energies,
\emph{J. Differential Equations} \textbf{262} (2017), no. 6, 3737--3758.

\bibitem{ortiz}
M. Ortiz, B. Schmidt, U. Stefanelli,
A variational approach to Navier-Stokes,
\emph{preprint}, arXiv:1802.06606 [math.AP] (2018).

\bibitem{serra}
E. Serra,
On a conjecture of De Giorgi concerning nonlinear wave equations,
\emph{Rend. Semin. Mat. Univ. Politec. Torino} \textbf{70} (2012), no. 1, 85--92.

\bibitem{ST1}
E. Serra, P. Tilli,
Nonlinear wave equations as limits of convex minimization problems: proof of a conjecture by De Giorgi,
\emph{Ann. of Math. (2)} \textbf{175} (2012), no. 3, 1551--1574.

\bibitem{ST2}
E. Serra, P. Tilli,
A minimization approach to hyperbolic Cauchy problems,
\emph{J. Eur. Math. Soc.} \textbf{18} (2016), no. 9, 2019--2044.

\bibitem{stefanelli}
U. Stefanelli,
The De Giorgi conjecture on elliptic regularization,
\emph{Math. Models Methods Appl. Sci.} \textbf{21} (2011), no. 6, 1377--1394.

\bibitem{strauss}
W. A. Strauss,
\emph{Nonlinear wave equations},
CBMS Regional Conference Series in Mathematics 73, AMS, Providence, RI, 1989.

\bibitem{struwe}
M. Struwe,
On uniqueness and stability for supercritical nonlinear wave and Schr\"odinger equations,
\emph{Int. Math. Res. Not.} (2006), Art. ID 76737, 14 pp.

\bibitem{tao}
T. Tao,
\emph{Nonlinear dispersive equations. Local and global analysis},
CBMS Regional Conference Series in Mathematics 106, AMS, Providence, RI, 2006.

\bibitem{tentarelli}
L. Tentarelli,
On the extensions of the De Giorgi approach to nonlinear hyperbolic equations,
\emph{Rend. Semin. Mat. Univ. Politec. Torino} \textbf{74} (2016), no. 2, 151--160.

\bibitem{TT}
L. Tentarelli, P. Tilli,
De Giorgi's approach to hyperbolic Cauchy problems: the case of nonhomogeneous equations,
\emph{Comm. Partial Differential Equations} published online (2018), DOI: 10.1080/03605302.2018.1459686.

\end{thebibliography}
\end{document}